\documentclass[11pt,twoside]{amsart}
\usepackage{amsmath, amsthm, amscd, amsfonts, amssymb, graphicx, color}
\usepackage[all]{xy}
\usepackage[bookmarksnumbered, plainpages]{hyperref}
\usepackage{graphicx}
\usepackage{tikz}
\usepackage{float}
\usepackage{wrapfig}
\usepackage{array}
\usepackage{makecell}

\newtheorem{thm}{Theorem}[section]
\newtheorem{cor}[thm]{Corollary}
\newtheorem{lem}[thm]{Lemma}
\newtheorem{prop}[thm]{Proposition}

\newtheorem{rem}[thm]{\bf{Remark}}

\newtheorem{quest}[thm]{Question}

\newtheorem{conjecture}[thm]{Conjecture}

\newcommand{\fit}{{\rm{Fit}}}
\newcommand{\aut}{{\rm{Aut}}}
\newcommand{\eng}{\mathord{\mathrm{eng}}}
\newcommand{\nil}{\mathord{\mathrm{nil}}}
\numberwithin{equation}{section}
\def\pn{\par\noindent}

\begin{document}


\title{Group nilpotency from a graph point of view}
\author{Valentina Grazian, Andrea Lucchini and Carmine Monetta}

\thanks{{\scriptsize
\hskip -0.4 true cm MSC(2020): Primary: 20D15; Secondary: 05C25.
}}
\maketitle

\begin{abstract}  
Let $\Gamma_G$ denote a graph associated with a group $G$. A compelling question about finite groups asks whether or not a finite group $H$ must be nilpotent provided $\Gamma_H$ is isomorphic to $\Gamma_G$ for a finite nilpotent group $G$. In the present work we analyze the problem for different graphs that one can associate with a finite group, both reporting on existing answers and contributing to new ones.
\end{abstract}

\vskip 0.2 true cm


\pagestyle{myheadings}
\markboth{\rightline {\sl Int. J. Group Theory x no. x (201x) xx-xx \hskip 7 cm  V. Grazian, A. Lucchini and C. Monetta}}
         {\leftline{\sl Int. J. Group Theory x no. x (201x) xx-xx \hskip 7 cm  V. Grazian, A. Lucchini and C. Monetta}}

\bigskip


\section{\bf Introduction}

Given a finite group $G$, one can consider a graph $\Gamma_G$ associated with $G$ which encodes certain group properties of $G$. Such an approach has been extensively studied in the last decades, mainly for two reasons. The former is to determine structure description of $G$ investigating the invariants of $\Gamma_G$, while the latter aims to produce graphs fitting specific features (see for instance \cite{ADM, az, cam, GM, LM, Lucc.Nem.}).

A natural question in this research line is to understand if a graph isomorphism - which is clearly a weaker relation than a group isomorphism - may or may not preserve specific properties of a group.  More precisely, we are interested in the following question. 

\begin{quest}\label{q1}
If $G$ and $H$ are finite groups with isomorphic graphs $\Gamma_G \cong \Gamma_H$ and $G$ is nilpotent, is it true that $H$ is nilpotent as well?
\end{quest}

Of course the hardness of the problem, as well as the answer, change depending on the graph choice. For instance we may easily produce a negative answer to Question \ref{q1} considering the {\it soluble graph} of a group $G$, that is the graph whose set of vertices is $G$ and in which two vertices are adjacent if they generate a soluble group. Indeed, both the soluble graph of a nilpotent group, and the soluble graph of a soluble group are complete, thus in this case Question \ref{q1} has negative answer whenever we pick a nilpotent group $G$ and a non-nilpotent soluble group $H$ having the same order: the smallest example is given by the cyclic group of order $6$ and the symmetric group of degree $3$.

However, our report on this problem will clarify that the situation is not so easy in general and that in some cases Question \ref{q1} remains still open.

In the following we will analyze the distinct circumstances corresponding to the following graph choices: the non-commuting graph, the power graph, the prime graph (also known as Gruenberg-Kegel graph), the generating %
graph, the Engel graph and the join graph; our study is summarized in Table \ref{table}.

\begin{center}
\begin{table}[H]\label{table} 
\begin{tabular}{ |c|c|c|m{3cm}| }
 \hline

 & Answer to  Q \ref{q1} & Positive answer if & Examples with $H$ non-nilpotent   \\
 \hline
 
 Power graph & Yes (Thm \ref{power.YES}) & & \\
\hline

Engel graph & Yes (Prop \ref{zorn}) & & \\ 
\hline

Non-commuting graph & Open & \thead{$|G| = |H|$ (Thm \ref{nc.same.order}) 
\\
or
\\
$G$ is an AC-group and $|Z(G)|\geq|Z(H)|$ \\ (Thm \ref{nc.AC.groups})} &\\
\hline

Generating graph & Open & \thead{$H$ is supersoluble (Thm \ref{super})
\\
or
\\
The subgraph of the non-generating graph \\obtained  by removing 
all universal vertices \\ is disconnected 
(Thm \ref{q1.ng.disconnected})} & \\
\hline


Prime graph & No &  \thead{$|H|$ is square-free (Prop \ref{prime.squarefree})} & $G\cong C_6 \times C_6$ and $H\cong S_3 \times C_6$ \\
\hline

Join graph & No & 
& $G\cong C_p \times C_p$ and $H\cong D_{2p}$, for $p$ an odd prime\\
 \hline

\end{tabular}
\medskip
\caption{Answers to Question \ref{q1} depending on the graphs.}
\end{table}
\end{center}

\vspace{-1.2cm}
We will mainly adopt standard terminology in graph theory and known notation in group theory. If $G$ is a finite group, $|G|$ denotes the order of $G$, $Z(G)$ denotes the center of $G$, $\fit(G)$ stands for the Fitting subgroup of $G$ and $\Phi(G)$ denotes the Frattini subgroup of $G$. Also, we will write $C_n$, $S_n$, $D_{2n}$ and $Q_n$ for the cyclic group of order $n$, the symmetric group of degree $n$, the dihiedral group of order $2n$ and the quaternion group of order $n$, respectively. 



 \section{The power graph and the enhanced power graph} 
 The power graph of a semigroup was first introduced by  Chakrabarty et al. \cite{Ch.power}, taking inspiration from the definition of directed power graph given by Kelarev and Quinn \cite{KQ.power}.
 The power graph of a finite group 
$G$ is a graph whose vertices are the elements of $G$ and in which two distinct vertices $x$ and $y$ are adjacent if there exists $k\geq 2$ such that $x^k = y$ or $y^k = x$. The enhanced power graph of $G$, also known as cyclic graph of $G$, is the graph with vertex set $G$ in which distinct elements $x$ and $y$ are joined by an edge if and only if $\langle x,y\rangle$ is cyclic. Note that the power graph is a subgraph of the enhanced power graph, and the two graphs coincide if and only if any element of the group has prime power order (see Theorem 28 of \cite{AACNS}). 
However, Question \ref{q1} is equivalent for power graphs and enhanced power graphs, thanks to the following result:

\begin{thm}\cite[Corollary 3.1]{ZBM}
Two finite groups $G$ and $H$ have isomorphic power graph if and only if they have isomorphic enhanced power graph.
\end{thm}

In \cite{CG}, Cameron proved that if $G$ and $H$ are finite abelian groups with isomorphic power graph, then they are isomorphic (see \cite[Theorem 1]{CG}). Then in \cite{cameron}, the author proved that groups with isomorphic power graph have the same numbers of elements of each order (see \cite[Corollary 3]{cameron}). Thanks to this result, Mirzargar and Scapellato managed to give a positive answer to Question \ref{q1} for power graphs:

\begin{thm}\cite[Corollary 3.2]{MS}\label{power.YES}
Let $G$ and $H$ be finite groups with isomorphic power graphs. If $G$ is nilpotent, then $H$ is nilpotent.   
 \end{thm}

 \begin{proof}
Let $p$ be a prime dividing $|G|$, let $P$ be the unique Sylow $p$-subgroup of $G$ and set $|P|=p^n$. Since $G$ and $H$ have isomorphic power graph, we get $|G|=|H|$ and so the Sylow $p$-subgroups of $H$ must have order $p^n$. Also, $G$ has exactly $p^n$ elements of $p$-power order and by  \cite[Corollary 3]{cameron}, the graph isomorphism implies that $H$ has exactly $p^n$ elements of $p$-power order too. Thus we conclude that $H$ has a unique Sylow $p$-subgroup. As this holds for every prime divisor of $|H|$, we deduce that $H$ is nilpotent.
 \end{proof}

Unfortunately, the proof of the previous theorem does not allow to obtain an explicit criterion to detect the nilpotency of a finite group $G$ looking at the associated power graph. The situation is much better if one considers the directed power graph associated to $G$: in this case  the vertex set is again $G$ and there is an arc $y\mapsto x$ if $x$ is a power of $y$. In particular the number of vertices $y\neq x$ such that $x\mapsto y$ is $|x|-1$, so $|x|$ is determined by the graph. Thus just looking at the direct power graph of $G$ we may compute, for any prime divisor $p$ of $|G|$, the number of $p$-elements of $G$ and deduce immediately whether the Sylow $p$-subgroup is normal in $G$.


\section{The Engel graph}

Following a suggestion given by Cameron (see \cite[Section 11.1]{cam}), we may define a graph $\Gamma_{\eng}(G)$, where the vertices are the elements of $G$ and where two vertices are adjacent if they satisfy a suitable Engel relation; more precisely if $x$ and $y$ are different elements of $G,$ then there is an edge joining $x$ and $y$ if and only if either $[x,{}_ry]=1$ or $[y,{}_rx]=1$ for some $r\in \mathbb N$. Cameron proposes to call this graph the Engel graph of $G$, although, as he notices, the same term was used by 
Abdollahi \cite{abd} to denote a related but different graph. 	

It is easy to see that the answer to Question \ref{q1} is affirmative for Engel graphs:

\begin{prop}\label{zorn}
A finite group $G$ is nilpotent if and only if its Engel graph is complete. In particular, if $G$ and $H$ are finite groups with isomorphic Engel graphs and $G$ is nilpotent, then $H$ is nilpotent.
\end{prop}

\begin{proof}
If $G$ is nilpotent, then every element of $G$ is a right Engel element. Hence the graph  $\Gamma_{\eng}(G)$ is complete. On the other hand, if the graph $\Gamma_{\eng}(G)$ is complete, then by \cite[12.3.4]{rob}, we conclude that $G$ is nilpotent.
\end{proof}

Cameron proposed to investigate the relation between the Engel graph and the Nilpotent graph (where the Nilpotent graph $\Gamma_{\nil}(G)$ has as vertices the elements of $G$ and $x$ and $y$ are adjacent if and only if $\langle x,y\rangle$ is nilpotent). In particular he asks (see \cite[Question 24]{cam}) for which groups $G$ the two graphs coincide. We answer to this question with the following result.
	
\begin{thm}
	Let $G$ be a finite group. Then  $\Gamma_{\eng}(G)=\Gamma_{\nil}(G)$ if and only if $G$ is nilpotent.
\end{thm}
\begin{proof}
	We prove by induction on the order of $G$ that if $\Gamma_{\eng}(G)=\Gamma_{\nil}(G)$, then $G$ is nilpotent. The property that $\Gamma_{\eng}(G)=\Gamma_{\nil}(G)$ is inherited by all the subgroups of $G.$ So by induction all the proper subgroups of $G$ are nilpotent, and this implies that $G$ is soluble. The set of universal vertices of the graph  $\Gamma_{\nil}(G)$ coincides with the hypercenter of $G$ (see \cite[Proposition 2.1]{az}), while the set of universal vertices of  $\Gamma_{\eng}(G)$ is the set of elements that are either left or right Engel and coincides with the Fitting subgroup of $G$ (see e.g. \cite[12.3.7]{rob}). Since $\Gamma_{\eng}(G)=\Gamma_{\nil}(G)$, if follows that  $G$ is a finite soluble group, whose hypercenter and Fitting subgroup $\fit(G)$ coincide, and this is possible only if $G$ is nilpotent.
Indeed, arguing by contradiction, assume that $G$ is not nilpotent. Since $G$ is soluble and $\fit(G) <  G$ by assumption, there exists a normal subgroup $N$ of $G$ such that $1 \neq N/\fit(G) = \fit(G/\fit(G))$. In particular, there is an integer $h\geq 2$ such that the $h$-term of the lower central series of $N$, denoted $\gamma_h(N)$, is contained in $\fit(G)$. Since $\fit(G)$ coincides with the hypercenter of $G$, we can find an integer $k\geq 2$ such that $\gamma_{h+k}(N) =1$, proving that $N$ is nilpotent and contradicting the fact that $N \nleq \fit(G)$. This proves that $G$ must be nilpotent.

	Conversely, if $G$ is nilpotent, then every element of $G$ is a right Engel element, so $\Gamma_{\eng}(G)=\Gamma_{\nil}(G)$ is the complete graph on $|G|$ vertices.
\end{proof}

\section{The non-commuting graph}
The non-commuting graph of a group was first considered by Paul Erd\H{o}s in 1975, while stating a problem solved by Neumann in \cite{Neumann}. If $G$ is a finite group, the non-commuting graph of $G$ is the graph whose vertices are the non-central elements of $G$ (i.e. $G \backslash Z(G)$) and in which two vertices $x$ and $y$ are adjacent if and only if they do not commute, or equivalently, if the group $\langle x,y\rangle$ is non-abelian. Note that the non-commuting graph of $G$ is the complement of the commuting graph of $G$ (where two elements are joined if they commute). 

 Question \ref{q1} in terms of the non-commuting graph was first posed by Abdollahi, Akbari and Maimani in \cite{AAM}, and it is in fact still open. 
Nevertheless, Question \ref{q1} is known to be true under certain extra conditions.

\begin{thm}\cite[Theorem 3.24]{AAM}\label{nc.same.order}
Let $G$ and $H$ be finite non-abelian groups with isomorphic non-commuting graphs. If $G$ is nilpotent and $|G| = |H|$, then H is nilpotent.
\end{thm}

\begin{proof}
Note that by the main result in \cite{CHM}, it is enough to prove that $G$ and $H$ have the same number of conjugacy classes of the same size. For every integer $i\geq 1$, let $m_i(G)$ and $m_i(H)$ denote the number of conjugacy classes of size $i$ of $G$
 and $H$, respectively. We will show that $m_i(G) = m_i(H)$ for every $i\geq 1$.
 First note that since the non-commuting graphs of $G$ and $H$ are isomorphic, we get $|G| - |Z(G)| = |H| - |Z(H)|$ (looking at the number of vertices). By assumption, $|G| = |H|$, so we deduce that $m_1(G) = |Z(G)| = |Z(H)| = m_1(H)$. Now, for every $g \in G \backslash Z(G)$, if  $h$ is the image of $g$ under the graph isomorphism, then looking at the number of vertices joined to $g$ and $h$, respectively, we deduce that $|G| - |C_G(g)| = |H| - |C_H(h)|$. Using again the hypothesis $|G|=|H|$, we obtain $|g^G| = [G \colon C_G(g)] = [H \colon C_H(h)] = |h^H|$. This implies that for every $i > 1$ we have $m_i(G) = m_i(H)$, concluding the proof. 
 \end{proof}

 In the same work, Abdollahi, Akbari and Maimani conjectured that any two finite groups having isomorphic non-commuting graph should have the same order, but this was proven to be false in \cite{counter.non.comm}. The counterexample built in \cite{counter.non.comm} involves two finite groups that are nilpotent, and so it does not affect Question \ref{q1}. Also, such groups are finite AC-groups, that is, finite groups whose centralizers of non-central elements are all abelian. Taking inspiration from such result, in a recent paper Grazian and Monetta proved the following:

 \begin{thm}\label{nc.AC.groups}\cite[Corollary 1.4]{GM}
Let $G$ and $H$ be finite non-abelian groups with isomorphic non-commuting graphs. If $G$ is a nilpotent AC-group and $|Z(G)| \geq |Z(H)|$, then $H$ is nilpotent.
 \end{thm}

 \begin{proof}[Idea of the proof]
 First notice that the isomorphism between the non-commuting graphs of the finite groups $G$ and $H$ implies that $H$ is an AC-group. By \cite[Proposition 3.14]{AAM}, if $H$ is non-solvable, then $|G|=|H|$ and we conclude by Theorem \ref{nc.same.order}. Therefore we can assume that $H$ is a solvable AC-group. Such groups have been classified in \cite[Satz 5.12]{Schmidt1970}, showing that if $H$ is non-abelian then $H$ must follow one of $4$ possible characterizations. The proof is completed analyzing the different possibilities.     
 \end{proof}

 Observe that a nilpotent AC-group has a unique Sylow subgroup that is non-abelian. Question \ref{q1} in the case of a finite nilpotent group $G$ with at least two distinct non-abelian Sylow subgroups has also been covered:

 \begin{thm}\cite[Theorem 2.4]{non-comm.nilp}\label{nc.two.sylow}
Let $G$ and $H$ be finite non-abelian groups with isomorphic non-commuting graphs. If $G$ is nilpotent, $G$ has at least two non-abelian Sylow subgroups and $|Z(G)| \geq |Z(H)|$, then $|G|=|H|$.    
 \end{thm}

 Note that if the assumptions of Theorem \ref{nc.two.sylow} are satisfied, then once again the conclusion that $H$ is nilpotent is reached by Theorem \ref{nc.same.order}.

 Theorems \ref{nc.AC.groups} and \ref{nc.two.sylow} imply that in order to give a positive answer to Question \ref{q1} for non-commuting graphs with the additional assumption $|Z(G)| \geq |Z(H)|$, it remains to consider the case in which $G$ is a finite nilpotent group of the form  $G = P \times A$, for a non-abelian $p$-group $P$ and an abelian group $A$ with $(|P|,|A|)=1$, containing at least one element $x \in G \backslash Z(G)$ such that $C_G(x)$ is not abelian. In this view, Grazian and Monetta posed the following conjecture:

 \begin{conjecture}\cite[Conjecture 3]{GM}\label{conj.nostra}
 Let $p$ be a prime and suppose $G=P \times A$ is a finite group where $P \in \rm{Syl}_p(G)$ is non-abelian and $A$ is an abelian $p'$-group. If $H$ is a finite group whose non-commuting graph is isomorphic to the one of $G$ and $|Z(G)| \geq |Z(H)|$ then $H=Q \times B$, where $q$ is a prime, $Q \in \rm{Syl}_q(H)$ is non-abelian and $B$ is an abelian $q'$-group. In particular, $H$ is nilpotent.
 \end{conjecture}

 \begin{rem}
 Theorem \ref{nc.AC.groups} implies that Conjecture \ref{conj.nostra} is true when $G$ is an AC-group. Also, if $G$ is a finite $p$-group (so $A=1$), then $|G|=|H|$ by \cite[Theorem 1.2]{non.comm.p.groups}, and so Conjecture \ref{conj.nostra} is true in this instance too.  
 \end{rem}

\section{The generating graph}

The generating  graph of a finite $2$-generated group $G$ is the graph defined on the elements of $G$ in such a way that two distinct vertices are connected by an edge if and only if they generate $G$.  It was defined by Liebeck and
Shalev in \cite{LS}, and has been further investigated by
many authors: see for example \cite{ bglmn,  bucr, tra, CL3,  fu, diam, LM2, lm3, LMRD} for some
of the range of questions that have been considered.

Question \ref{q1} is still open for generating graphs (even with the extra assumption that $H$ is solvable) and appears to be a difficult problem.

In this section, we prove that the answer to Question \ref{q1} is affirmative at least in the particular case when $H$ is a finite supersoluble group: 

\begin{thm}\label{super}
Let $G$ and $H$ be finite $2$-generated groups with isomorphic generating graphs. If $G$ is nilpotent and $H$ is supersoluble, then $H$ is nilpotent.
\end{thm}

First we need an easy numerical lemma.

\begin{lem}\label{unfat}Let $\alpha=(a_1,\dots,a_r)$ and  $\beta=(b_1,\dots,b_s)$ be two sequences of prime numbers, with $a_1\leq \dots \leq a_r$ and
	$b_1\leq \dots \leq b_s.$
	If
	$$\prod_{1\leq i\leq r}\left(1-\frac{1}{a_i}\right)=\prod_{1\leq j\leq s}\left(1-\frac{1}{b_j}\right),$$ then $\alpha=\beta.$
\end{lem}
\begin{proof}
	By induction on $r+s.$ If $r+s=2$, then $r=s=1$ and the statement is trivial. So suppose $r+s > 2$. We have
	\begin{equation}\label{rs}
		\prod_{1\leq i\leq r} a_i\prod_{1\leq j\leq s}(b_j-1)=\prod_{1\leq i\leq r} (a_i-1)\prod_{1\leq j\leq s}b_j.
	\end{equation}
	Let $p=\max\{a_1,\dots,a_r,b_1,\dots,b_s\}$, $r^*=\max\{i \mid a_i\neq p\}$, $s^*=\max\{j \mid b_j\neq p\}$. 
 Note that we can assume that both $r^*$ and $s^*$ exist. Indeed, if for example $a_i = p$ for every $1 \leq i \leq r$, then from (\ref{rs}) we get 
 \begin{equation*}
		p^r\prod_{1\leq j\leq s}(b_j-1)= (p-1)^r\prod_{1\leq j\leq s}b_j;
	\end{equation*} 
 thus $\beta$ contains exactly $r$ primes equal to $p$ and if $\alpha \neq \beta$ then
 	$\prod_{b_j \neq p}(b_j-1)=\prod_{b_j \neq p}b_j$,
a contradiction.

 Now, since $p$ does not divide $a_i-1$ nor $b_j-1$, it divides $a_i$ if and only if $i>r^*$ and divides $b_j$ if and only if $j>s^*.$ We deduce that $r-r^*$ is the multiplicity of $p$ in the left term of (\ref{rs}) and $s-s^*$ is the multiplicity of $p$ in the right term of (\ref{rs}). In particular $r-r^*=s-s^*$ and $a_{r^*+1}=\dots=a_r=b_{s^*+1}=\dots=b_s=p.$ But then
	$$\prod_{i\leq r^*}\left(1-\frac{1}{a_i}\right)=\prod_{j\leq s^*}\left(1-\frac{1}{b_j}\right),$$ and we conclude by induction.
\end{proof}

Now we need some information on the degrees of the vertices of the generating graph of a finite nilpotent group. From now on, let $\Gamma(G)$ denote the generating graph of $G$.

\begin{lem}\label{gradinilpo}
Let $G$ be a 2-generated, non cyclic,  finite nilpotent group. Let $\pi(G)$ be the set of prime divisors of $G,$ $\pi_1(G)=\{p_1,\dots,p_r\}$ the set of the primes $p\in \pi(G)$ such that the Sylow $p$-subgroup of $G$ is cyclic and $\pi_2(G)=\{q_1,\dots,q_s\}$ the set of the remaining primes. For every subset $I$ of $\{1,\dots,r\}$, let 	$$\begin{aligned}\alpha_I&=|G|\prod_{1\leq j\leq s}\left(1-\frac{1}{q_j^2}\right)\prod_{i \in I}\left(1-\frac{1}{p_i}\right)\prod_{i \notin I}\frac{1}{p_i},\\ \beta_I&=|G|\prod_{1\leq j\leq s}\left(1-\frac{1}{q_j}\right)\prod_{i \notin I}\left(1-\frac{1}{p_i}\right).
	\end{aligned}$$

For $g\in G$, denote by $\delta_G(g)$ the degree of $g$ in the generating graph of $G$. If $g$ is a non-isolated vertex, then $\delta_G(g)=\beta_I$ for some subset $I$ of $\{1,\dots,r\}$. Moreover for every $I\subseteq \{1,\dots,r\}$, the generating graph of $G$ contains precisely $\alpha_I$ vertices of degree $\beta_I.$
\end{lem}

\begin{proof} An element $g\in G$ is not isolated in $\Gamma(G)$ if and only if $g\Phi(G)$ is not isolated in $\Gamma(G/\Phi(G))$, and this occurs if and only if $q_1\cdots q_s$ divides $|g\Phi(G)|$.
Given $I\subseteq \{1,\dots,r\},$ there are precisely $\alpha_I$ elements $g$ such that $|g\Phi(G)|=q_1\cdots q_s\prod_{i \in I}p_i.$ All these elements have degree $\beta_I$.
\end{proof}

Now we collect some information on the generating graph of a 2-generated supersoluble group. 

\begin{lem}\label{supuno}
Assume that $X$ is a finite 2-generated supersoluble group and that $\Phi(X)=1.$ Then
$$X \cong (V_1\times \dots \times V_t)\rtimes Y$$ where $Y$ is abelian and $V_1,\dots,V_r$ are pairwise non $Y$-isomorphic nontrivial irreducible $Y$-modules.
\end{lem}

\begin{proof}
The Fitting subgroup $\fit(X)$ of $X$ is a direct product of minimal normal subgroups of $X$, and is complemented in $X$. Let $T$ be a complement of $\fit(X)$ in $X$. Since $X$ is supersoluble, $X^\prime \leq \fit(X),$ and consequently $T$ is abelian. 
Let $W$ be a complement of $Z(X)$ in $\fit(X).$ Then $X=W \rtimes Y,$ with $Y=\langle T, Z(X)\rangle.$ We may decompose $W=V_1^{n_1}\times \dots V_t^{n_t},$ 
where	$V_1,\dots,V_r$ are pairwise non $Y$-isomorphic nontrivial irreducible $Y$-modules. The condition that $X$ is 2-generated implies that $n_1=\dots =n_t=1.$
\end{proof}

\begin{lem}\label{supdue} Let $X= (V_1\times \dots \times V_t)\rtimes Y$ be as in Lemma \ref{supuno}. For $1\leq i\leq t,$ let $|V_i|=r_i$ (since $X$ is supersoluble, $r_i$ is a prime). Assume that $x=(v_1,\dots,v_t)y$ is a non-isolated vertex of the generating graph of $X$ and let $J_y=\{j\in \{1\dots,t\}\mid [y,V_j]\neq 0\}.$
	Then
	\begin{equation}\label{grad}\delta_X(x)=\delta_Y(y)\prod_{j\notin J_y}r_j\prod_{j\in J_y}(r_j-1).\end{equation}
\end{lem}

\begin{proof}
For $1\leq i \leq t,$ we may identify $V_i$ with the additive group of the field $F_i$ with $r_i$ elements. For every $z\in Y$ and $1\leq i \leq t,$ there exists $\alpha_i(z)\in F_i$ such that $w_i^z=\alpha_i(z)w_i$ for all $w_i\in V_i.$
Let $\tilde x=(\tilde v_1,\dots,\tilde v_t)\tilde y \in X.$ 
It follows from Propositions 2.1 and
2.2 in \cite{LM2}, that
 $\langle x, \tilde x\rangle =X $ if and only if $\langle y, \tilde y\rangle =Y$
 and
$$\delta_i(x,\tilde x):=\det\begin{pmatrix}1-\alpha_i(y)&1-\alpha_i(\tilde y)\\v_i&\tilde v_i
\end{pmatrix}\neq 0 \quad \text { for all $i\in \{1,\dots,t\}$}.$$
If $i\notin J_y,$ then $\alpha_i(y)=1$. Since $\langle y, \tilde y\rangle=Y$ and $[Y,V_i]=V_i,$ it must be that $\alpha_i(\tilde y)\neq 1$,
and therefore $\delta_i(x,\tilde x)\neq 0$
if and only if $v_i\neq 0,$ independently on the choice of $\tilde v_i.$
If $i\in J_y,$ then, for every choice of $v_i$ and $y,$ the probability that $\tilde v_i$ satisfies the condition
$\det\begin{pmatrix}1-\alpha_i(y)&1-\alpha_i(\tilde y)\\v_i&\tilde v_i
\end{pmatrix}\neq 0$
coincides with $1-1/r_i$.

 We conclude that $x$ is not isolated in $\Gamma(X)$ if and only if $v_i\neq 0$ for every $i \notin J_y.$ Moreover if $x$ is not isolated, then (\ref{grad}) holds.
		\end{proof}

\begin{proof}[Proof of Theorem \ref{super}]
If $H$ is cyclic, then $H$ is nilpotent and there is nothing to prove. So assume $H$ is not cyclic	
and let $n=|G|=|H|$. Moreover assume that $\pi_1(G)$ and $\pi_2(G)$ are as described in Lemma \ref{gradinilpo} and that $X=H/\Phi(H)$ is as described in Lemmas \ref{supuno} and \ref{supdue}.

For every element $h\in H$, we write $h\Phi(H)=w_hy_h$ with
$w_h \in V_1\times \dots \times V_t$ and $y_h\in Y$. Since $\langle h_1,h_2\rangle=H$ if and only if $\langle h_1,h_2\rangle\Phi(H)=H$,
it follows from Lemma \ref{supdue} that
	\begin{equation}\label{grado}\delta_H(h)=\frac{n\cdot \delta_Y(y_h)}{|Y|}\prod_{j\notin J_{y_h}}r_j\prod_{j\in J_{y_h}}(r_j-1).
\end{equation} 
To reach our conclusion, we will prove a series of consecutive claims.

\

\noindent Claim 1. {\sl{Let $y\in Y$ be such that $Y/\langle y\rangle$ is cyclic and let $\omega$ be the set of prime divisors of $|Y/\langle y\rangle|$. Then $\{r_j \mid j\in J_y\} \cap \omega = \emptyset.$ Moreover
if $j_1, j_2 \in J_y$ and $j_1\neq j_2,$ then $r_{j_1}\neq r_{j_2}.$
}}

Indeed there exists a non isolated vertex $h$ in $\Gamma(H),$ with $y=y_h,$ and by (\ref{grado}),
\begin{equation}\label{gradi2}
	\delta_H(h)=n\prod_{j\in J_y}\left(1-\frac{1}{r_j}\right)\prod_{u\in \omega}\left(1-\frac{1}{u}\right).\end{equation}
Since $\Gamma(G)\cong \Gamma(H)$, there exists $g\in G$ with $\delta_H(h)=\delta_G(g).$ 
It follows from Lemma \ref{gradinilpo}, that 
there exists $I\subseteq \{1, \dots, r\}$ such that
\begin{equation}\label{confronto}\prod_{j\in J_y}\left(1-\frac{1}{r_j}\right)\prod_{u\in \omega}\left(1-\frac{1}{u}\right)=\prod_{1\leq j\leq s}\left(1-\frac{1}{q_j}\right)\prod_{i \notin I}\left(1-\frac{1}{p_i}\right).\end{equation}
The factors in the right term of  (\ref{confronto}) are all distinct. By Lemma \ref{unfat} the same must be true for the left term, and this implies that Claim 1 is true.

\

\noindent Claim 2. {\sl{The prime $r_i$ does not divide $|Y|,$ for every $1\leq i\leq t.$}}

Indeed, assume by contradiction that $r_i$ divides $|Y|.$ Since $Y$ is a 2-generated abelian group and $Y/C_Y(V_i)\leq \aut(V_i)$ is cyclic of order dividing $r_i-1,$ there exists $y\in Y$ such that $i\in J_y$, $Y/\langle y\rangle$ is cyclic
and $r_i$ divides $|Y/C_Y(V_i)|,$ in contradiction with Claim 1.

\

\noindent Claim 3. {\sl{If $1\leq i < j\leq t,$ then $r_i\neq r_j$}}. 

Assume  $i\neq j.$ We can find $y\in Y$ such that
$\{i,j\}\subseteq  J_y$ and $Y/\langle y \rangle$ is cyclic (indeed let $Y=\langle y_1, y_2 \rangle$: at least one of the three elements $y_1, y_2, y_1y_2$ does not centralizes neither $V_i$ nor $V_j$ and this is the element we need). By Claim 1 it must be $r_i\neq r_j.$

\

Now let $\pi_1=\{r_1,\dots,r_t\},$ $\pi_2$ be the set of the prime  divisors $p$ of $|Y|$ such that the Sylow $p$-subgroup of $Y$ is cyclic, $\pi_3$ the set of the remaining prime divisors of $|Y|.$ We have proved that $\pi=\pi(G)=\pi(H)$ is the disjoint union of $\pi_1,$ $\pi_2$ and $\pi_3.$

\

In particular it follows from (\ref{grado}) and Claim 3 that every non isolated vertex $h\in \Gamma(H)$ uniquely determines a subset $\pi_h$ of $\pi$ such that $$\delta_H(h)=n\prod_{p\in \pi_h}\left(1-\frac{1}{p}\right).$$ Since the degrees in $\Gamma(G)$ and $\Gamma(H)$ are the same, denoting by $\Lambda(H)$ the set of the non isolated vertices of $\Gamma(H),$ 
it follows from Lemma \ref{gradinilpo} that
\begin{equation}\label{interseco}
\pi_2(G)=\bigcap_{h\in \Lambda(H)}\pi_h.
\end{equation}
Let $r_i\in \pi_1.$ Since $\aut(V_i)$ is cyclic, there exists $y\in Y$ such that $i\notin J_y$ and $Y/\langle y\rangle$ is cyclic. Moreover there exists $h\in \Lambda(H)$ with $y_h=y.$  By (\ref{gradi2}) and Claim 2, $r_i\notin \pi_h,$ and therefore, by (\ref{interseco}), $r_i\notin \pi_2(G).$ Hence
\begin{equation}\label{11}
	\pi_1\subseteq  \pi_1(G).
\end{equation}
If $r\in \pi_2,$ then there exists $y\in Y$ such that $Y/\langle y\rangle$ is cyclic and has order coprime with $r.$ As before, take $h\in \Lambda(H)$ with $y_h=y.$ By (\ref{gradi2}) and Claim 2, $r\notin \pi_h,$ hence $r\notin \pi_2(H).$ It follows that
\begin{equation}\label{12}
	\pi_2\subseteq  \pi_1(G).
\end{equation}
It follows easily from  (\ref{gradi2}), that if $p \in \pi_3,$ then $p \in \pi_h$ for every  $h\in \Lambda(H),$  hence
\begin{equation}\label{13}
	\pi_3 \subseteq \pi_2(G).
\end{equation}
Since $\pi$ is the disjoint union of $\pi_1, \pi_2, \pi_3,$ but also the disjoint union of $\pi_1(G)$ and $\pi_2(G),$ combining $(\ref{11}),$ $(\ref{12})$
and $(\ref{13}),$ we conclude 
$$\pi_1(G)=\pi_1\cup \pi_2 {\text { and }} \pi_2(G)=\pi_3.$$
Let $\eta_1, \eta_2$ be, respectively, the number of edges of $\Gamma(G)$ and $\Gamma(H).$  Notice that $2\eta_1=|G|^2P_G(2)$ and $2\eta_2=|H|^2P_H(2)$, where, given a 2-generated finite group $T$, we denote by $P_T(2)$ the probability that a pair of uniformly randomly chosen elements of $T$ generates $T$. Since $\Gamma(G)\cong \Gamma(H),$ we must have that $\eta_1=\eta_2$ and consequently $P_G(2)=P_H(2).$ It follows from \cite[Satz 4]{eul} that
$$\begin{aligned}P_H(2)&=\prod_{p\in \pi_1}\left(1-\frac{1}{p}\right)
\prod_{p\in \pi_2}\left(1-\frac{1}{p^2}\right)\prod_{p\in \pi_3}\left(1-\frac{1}{p}\right)\left(1-\frac{1}{p^2}\right),\\
P_G(2)&=
\prod_{p\in \pi_1(H)}\left(1-\frac{1}{p^2}\right)\prod_{p\in \pi_2(G)}\left(1-\frac{1}{p}\right)\left(1-\frac{1}{p^2}\right).\end{aligned}$$
Since $\pi_1\cup \pi_2=\pi_1(G)$ and $\pi_3=\pi_2(G),$  we deduce
$$\prod_{p\in \pi_1}\left(1+\frac{1}{p}\right)=1$$ and consequently $\pi_1=\emptyset.$ Therefore $H/\Phi(H)\cong Y$ is abelian and we conclude that $H$ is nilpotent.
\end{proof}

\subsection{An interesting subgraph}
In this subsection we aim to obtain some indications that information on the structure of a finite group $G$ can be obtained investigating the connectivity properties of the complement of the generating graph of $G$. Notice that the isolated vertices of the generating graph are joined to all the other vertices of the complement graph, so the complement of the generating graph is always connected. But an interesting graph is obtained if we consider the complement  $\Delta(G)$ of the subgraph of the generating graph induced by the subset of its non-isolated vertices. In other word the vertices of $\Delta(G)$ are the elements $x \in G$ with the property that $\langle x, y\rangle=G$ for some $y\in G$ and there is an edge between $x_1$ and $x_2$ if and only if $\langle x_1, x_2\rangle \neq G.$ Note that the vertex set of $\Delta(G)$ is contained in the set $G \backslash \Phi(G)$, but in general it can be smaller. For example, in $\Delta(S_4)$, the elements $(12)(34)$, $(13)(24)$ and $(14)(23)$ do not belong to the vertex-set.  We will show in particular that we have a positive answer to Question \ref{q1} with respect to the graph $\Delta(G)$, whenever it is disconnected.

We start analysing the graph $\Delta(G)$ when $G$ is either a cyclic group or a $p$-group.

\begin{lem}\label{ng.cyclic}
Let $G$ be a cyclic group of order $n$. Then $\Delta(G)$  contains $n$ vertices, and at least two of them are isolated.
\end{lem}

\begin{proof}
Suppose $G = \langle x \rangle$ is a cyclic group of order $n$. For every element $g \in G$, we have $\langle g, x \rangle = G$. Thus none of the elements of $G$ is a universal vertex, implying that $\Delta(G)$ contains $n$ vertices. Now, if $n=2$, then the elements $1$ and $x$ are isolated vertices, while for $n\geq 3$, the elements $x$ and $x^{-1}$ are distinct isolated vertices. So in any case there are at least two isolated vertices.
\end{proof}

\begin{lem}\label{ng.p.group}
   Let  $p$ be a prime and let $P$ be a $2$-generated $p$-group of order $p^n$ that is not cyclic. Then $\Delta(P)$ contains $p+1$ connected components, each one complete and containing $p^{n-2}(p-1)$ vertices. In particular $\Delta(P)$ is disconnected and contains isolated vertices if and only if $P \cong C_2 \times C_2$.
\end{lem}

\begin{proof}
Note that $P$ contains $p+1$ maximal subgroups $M_1, \dots M_{p+1}$ and the following holds:
\begin{itemize}
    \item for every $x,y \in M_i$, $x\neq y$, we have $\langle x, y \rangle \leq M_i < P $, so $x$ and $y$ are joined; 
    \item for every $x \in M_i$, $y \in M_j$ with $i\neq j$, we have $\langle x, y \rangle = P $, so $x$ and $y$ are not joined. 
\end{itemize}
Therefore $\Delta(P)$ contains $p+1$ connected components, each one complete and containing $|M_i| - |\Phi(P)| = p^{n-2}(p-1)$ vertices. In particular $\Delta(P)$ contains isolated vertices if and only if $p^{n-2}(p-1)=1$, that is equivalent to $P \cong C_2 \times C_2$. 
\end{proof}

\begin{lem}
Suppose $G$ is a cyclic group and $P$ is a $2$-generated $p$-group that is not cyclic, for some prime $p$. Then $\Delta(G) \cong \Delta(P)$ if and only if $G\cong C_3$ and $P \cong C_2 \times C_2$.  
\end{lem}

\begin{proof}
By Lemma \ref{ng.cyclic}, $\Delta(G)$ contains isolated vertices. Hence by Lemma \ref{ng.p.group} we deduce that $P \cong C_2 \times C_2$. In particular $\Delta(G) \cong \Delta(P)$ consists of $3$ vertices and no edges.  Hence again by Lemma \ref{ng.cyclic} we conclude that $G \cong C_3$.  
\end{proof}

Next, we consider the dihedral group $D_{2p}$, for $p$ an odd prime:

\begin{lem}\label{ng.dihedral}  
Let $p$ be an odd prime. Then $\Delta(D_{2p})$ has $D_{2p} \backslash \{1\}$ as vertex-set and consists of one connected component of size $p-1$ that is complete and $p$ isolated vertices (corresponding to the $p$ involutions of $D_{2p}$).  
\end{lem}

\begin{proof}
Note that the trivial element is a universal vertex (as $D_{2p}$ is non-cyclic) and it is the only one, so $\Delta(D_{2p})$ contains $D_{2p} \backslash \{1\}$ as vertex-set. Now, all involutions of $D_{2p}$ are isolated vertices, as each of them generates $D_{2p}$ together with any other non-trivial element. Finally, if $P\in {\rm{Syl}}_p(D_{2p})$, then for every pair of distinct non-trivial elements $x,y \in P$ we have $\langle x, y \rangle = P < D_{2p}$, and so the non-trivial elements of $P$ form a connected component of size $p-1$, that is complete. 
\end{proof}

Lucchini and Nemmi characterized all finite $2$-generated groups $G$ for which the graph $\Delta(G)$ contains isolated vertices:

\begin{prop}\cite[Proposition 2]{Lucc.Nem.}\label{isolated}
Let G be a $2$-generated finite group. Then $\Delta(G)$ has an isolated vertex if and only if one of the following holds:
\begin{enumerate}
\item $G$ is cyclic;
\item $G \cong C_2 \times C_2$; or
\item $G \cong D_{2p}$ for an odd prime $p$.  
\end{enumerate}
\end{prop}

As a consequence of Proposition \ref{isolated} we obtain the following:

\begin{cor}\label{cor.ng.isolated}
 Let $G$ and $H$ be finite $2$-generated groups with $\Delta(G) \cong \Delta(H)$. If $\Delta(G)$ contains isolated vertices and $G$ is nilpotent, then $H$ is nilpotent.   
\end{cor}

\begin{proof}
 By Proposition \ref{isolated}, the group $G$ is either cyclic or isomorphic to $C_2 \times C_2$ and it is enough to prove that $H$ cannot be isomorphic to the group $D_{2p}$. Aiming for a contradiction, suppose $H \cong D_{2p}$. If $G \cong C_2 \times C_2$, then $\Delta(H)$ should have $3$ vertices, and by Lemma \ref{ng.dihedral} we should have $2p -1 = 3$, that is impossible. Thus $G$ must be cyclic, say $G = \langle x \rangle$, and by Lemma \ref{ng.cyclic} we deduce that $|G| = 2p-1 \geq 5$. Also, by Lemma \ref{ng.dihedral} the graph $\Delta(G)$ has a complete connected component of size $p-1$ and $p$ isolated vertices. In particular $|G|$ is not prime (otherwise all vertices would be isolated). Suppose $q_1\neq q_2$ are prime numbers dividing $|G|$. Then both $x^{q_1}$ and $x^{q_2}$ are joined to the trivial element, but they are not joined among each other (as they generate $G$ together). However by Lemma \ref{ng.dihedral} $\Delta(D_{2p})$ cannot have this shape. Thus $G$ must be a cyclic $q$-group for some prime $q$, of order $q^k = 2p-1$ for $k\geq 2$. In particular the number of isolated vertices is $q^k -q^{k-1}$ and this must equal $p$.  Hence $q$ divides $p$, implying $q=p$ a contradiction.
 
 This proves that $H$ is either cyclic or isomorphic to  $C_2 \times C_2$, and so in particular it is nilpotent.      
\end{proof}

More generally, Lucchini and Nemmi described the finite $2$-generated groups $G$ in which $\Delta(G)$ is disconnected:

\begin{prop}\cite[Theorem 1]{Lucc.Nem.}\label{disconnected}
Let G be a $2$-generated finite group. Then $\Delta(G)$ is disconnected if and only if one of the following holds:
\begin{enumerate}
\item $G$ is cyclic;
\item $G$ is a $p$-group;
\item $G$ is not a $p$-group, $G/ \Phi(G) \cong (V_1 \times \dots \times V_t) \rtimes Y$, where $Y \cong C_p$ for some prime $p$, and $V_1, \dots, V_t$ are pairwise non-$Y$-isomorphic non-trivial irreducible $Y$-modules; or
\item $G$ is not a $p$-group, $G/ \Phi(G) \cong (V_1 \times \dots \times V_t) \rtimes Y$, where $Y \cong C_p \times C_p$ for some prime $p$, $V_1, \dots, V_t$ are pairwise non-$Y$-isomorphic non-trivial irreducible $Y$-modules and we have $C_Y(V_1 \times \dots \times V_t) \cong C_p$.
\end{enumerate}
\end{prop}

We are now ready to prove that Question \ref{q1} has a positive answer when $\Delta(G)$ is disconnected:

\begin{thm}\label{q1.ng.disconnected}
Let $G$ and $H$ be finite $2$-generated groups with $\Delta(G) \cong \Delta(H)$. If $\Delta(G)$ is disconnected and $G$ is nilpotent, then $H$ is nilpotent.
\end{thm}

\begin{proof}
Thanks to Corollary \ref{cor.ng.isolated}, we can assume that $\Delta(G)$ does not have isolated vertices, so in particular $G$ is not cyclic by Lemma \ref{ng.cyclic}.
By Proposition \ref{disconnected}, we can suppose that $G$ is a $p$-group of order $p^n$, for some prime $p$ and integer $n\geq 2$, and we have to show that $\Delta(G)$ is not isomorphic to $\Delta(X)$ whenever $X$ is a group of type (3) or (4).

By Lemma \ref{ng.p.group} $\Delta(G)$ contains $p^{n-2}(p^2-1)$ vertices and it is regular of valency $p^{n-2}(p-1) - 1$.

Aiming for a contradiction, suppose that $H$ is of type (3) or (4), with $K \cong C_q$  or $K \cong C_q \times C_q$ for some prime $q$. Note that in both cases $H$ contains a maximal subgroup $M$ that is normal in $H$ and such that $|H| = |M|q$. Indeed, if $H$ is of type (3) then $M/\Phi(H) \cong V_1 \times \dots \times V_t$, while if $H$ is of type (4) then $M/\Phi(H) \cong (V_1 \times \dots \times V_t) \times \langle k \rangle$ where $\langle k \rangle = C_K(V_1 \times \dots \times V_t)$. 
Let $\Omega$ denote the set of vertices of $\Delta(H)$ corresponding to elements of $M$. 
In \cite[Lemmas 24 and 25]{Lucc.Nem.} the authors studied the structure of $\Delta(H)$, proving that $\Omega$ is a proper connected component of $\Delta(H)$. In particular,
every vertex $x \in \Omega$ has degree $|\Omega| - 1$ (as $x$ is not joined to itself).
Also, if $h\in H$ is such that 
$H = M\langle h \rangle$, then every element of $H$ can be written as $mh^j$ for some $m\in M$ and $1\leq j \leq q$ and $\langle h, mh^j\rangle = \langle h , m \rangle < H$ if and only if $m \in M \backslash \Omega.$ Moreover, if $m \in M \backslash \Omega$, then $mh^j$ is a vertex if and only if $h^j\neq 1$ (that is equivalent to $j\neq q$). This implies that $h$ has degree $(|M| - |\Omega|)(q-1) - 1$. 

Since $\Delta(G)$ and $\Delta(H)$ are isomorphic, the graph $\Delta(H)$ must be regular. Hence
\[ |\Omega| - 1 = (|M| - |\Omega|)(q-1) - 1  \]
that gives

\[ |\Omega| = |M|\frac{q-1}{q}.  \]

Comparing the valencies of $\Delta(G)$ and $\Delta(H)$ we obtain
\begin{equation}\label{eq.val} p^{n-2}(p-1) = |\Omega| =  |M|\frac{q-1}{q}.\end{equation}

Also, the total number of vertices of $\Delta(H)$ is
\[|\Omega| + (|H| - |M|) = |\Omega| + |M|(q - 1) = |M|\frac{q-1}{q} + |M|(q - 1) = |M|(q-1)\left(\frac{1}{q} + 1\right). \]

Dividing the number of vertices of $\Delta(G)$ and $\Delta(H)$ by the values obtained in equation \ref{eq.val} we get
\[ \frac{p^{n-2}(p^2 -1)}{p^{n-2}(p-1)} = \frac{|M|(q-1)\left(\frac{1}{q} + 1\right) }{|M|\frac{q-1}{q}}.\]

Simplifying the above equation, we conclude that $p=q$. Finally, again by equation \ref{eq.val} we obtain $|M| = p^{n-1}$ and so $|H|= p|M| = p^n$ and $H$ is a $p$-group, a contradiction. This proves the statement.
\end{proof}

\section{The prime graph (or Gruenberg-Kegel graph)}

The prime graph (also known as Gruenberg-Kegel graph) was introduced by Gruenberg and Kegel in an unpublished paper in 1975.
For a finite group $K$, denote by $\pi(K)$ the set of prime divisors of the order of $K$. The prime graph of a finite group $G$ is the graph having $\pi(G)$ as vertex set and such that two distinct vertices $p$ and $q$ are adjacent if and only if $G$ contains an element of order $pq$. Note that if $p$ is a prime and $G$ is a finite $p$-group, then the prime graph of $G$ contains only one vertex and no edges. In particular Question \ref{q1} is trivially true for $p$-groups. However, it is not hard to see that it is false in general:

\begin{prop}
Let $G$ be a finite nilpotent group and suppose there exists a finite non-nilpotent group $K$ with $\pi(G)=\pi(K)$. Then $H:=K \times K$ is not nilpotent and the prime graphs of $G$ and $H$ are isomorphic.    
\end{prop}

\begin{proof}
Note that $H$ is not nilpotent, because $K$ is not nilpotent by assumptions, $\pi(H)=\pi(G)$ and the prime graph of $H$ is complete, and so isomorphic to the one of $G$.     
\end{proof}

We can even find finite groups with the same order representing a negative answer to Question \ref{q1}. Take for example $G = C_6 \times C_6$ and $H= S_3 \times C_6$, both having a complete prime graph on two vertices. However, Question \ref{q1} has a positive answer if we assume that $H$ has square-free order:

\begin{prop}\label{prime.squarefree}
Let $G$ and $H$ be finite groups with isomorphic prime graphs. If $G$ is nilpotent and $|H|$ is square-free, then $H$ is cyclic (hence nilpotent). 
\end{prop}

\begin{proof}
Note that the assumption that $|H|$ is square-free implies that every Sylow subgroup of $H$ is cyclic of prime order. If $H$ is a $p$-group for a prime $p$, we are done.  Assume that there exist distinct primes $p$ and $q$ dividing the order of $H$. By hypothesis, the prime graph of $H$ is isomorphic to the one of $G$, that is complete.  Therefore there exists an element $x$ in $H$ of order $pq$ and $P=\langle x^q \rangle$ is a Sylow $p$-subgroup of $H$. We show that $P$ is contained in the center of $H$, by proving that for any prime $r$ dividing the order of $H$ there exists a Sylow $r$-subgroup of $H$ commuting with $P$. If $r=q$, then $P$ commutes with $Q=\langle x^p \rangle \in {\rm{Syl}}_q(H)$. Hence assume that $r$ is distinct from $p$ and $q$. Then there exists an element $y \in H$ of order $pr$ and the Sylow $p$-subgroup $P_1=\langle y^r\rangle$ commutes with the Sylow $r$-subgroup $R= \langle y^p\rangle $. Since $P$ and $P_1$ are conjugate, it follows that $P$ commutes with a conjugate of $R$.
The arbitrary choice of $p$, shows that $H$ is isomorphic to the direct product of cyclic groups of coprime orders, hence it is cyclic.
\end{proof}

\section{The join graph}
We conclude our survey with a graph that focuses the attention on the subgroup lattice of a finite group. The join graph of a finite group $G$ has been introduced by Lucchini in \cite{Lucc} as follows: it is the graph having as vertex-set the set of proper subgroups of $G$ and in which two subgroups $M$ and $N$ are joined if and only if $G = \langle M, N \rangle$. In general Question \ref{q1} has a negative answer for the join graph: consider for example the groups $C_p \times C_p$ and $D_{2p}$, for $p$ odd. However the following holds:

\begin{thm}\cite[Corollary 5]{Lucc}\label{join.H.super}
    Let $G$ and $H$ be finite groups with isomorphic join graphs. If $G$ is nilpotent, then $H$ is supersoluble.
\end{thm}

We can say more if the Frattini subgroup of $H$ is trivial.

\begin{lem}\label{lem:frat}
Let $G$ and $H$ be finite groups with isomorphic join graphs. If $\Phi(H) =1$ then $\Phi(G) =1$. In particular, if $\Phi(H) =1$ and $G$ is nilpotent then G is a direct product of elementary abelian groups.
\end{lem}

\begin{proof}
Since $\Phi(H)=1$, the identity subgroup is the unique isolated vertex in the join graph of $H$. Using the graph isomorphism, we deduce that the join graph of $G$ has a unique isolated vertex that must correspond to the trivial subgroup, and so $\Phi(G)$ is trivial too. Finally, if $G$ is nilpotent then for every prime divisor $p$ of $|G|$, if $P\in \rm{Syl}_p(G)$ then $\Phi(P) \leq \Phi(G)=1$ and so $P$ is elementary abelian. So all Sylow subgroups of $G$ are elementary abelian and $G$ is a direct product of elementary abelian groups.  
\end{proof}

Recall that a finite group $K$ is a $P$-group if it is either non-cyclic elementary abelian or a semidirect product of an elementary abelian $p$-group $A$ by a group of prime order $q \neq p$ which induces a non-trivial power automorphism on $A$.
Using Lemma \ref{lem:frat} and properties of the lattice of $G$, Lucchini obtained the following result:

\begin{thm}\cite[Proposition 6]{Lucc}
Let $G$ and $H$ be finite groups with isomorphic join graphs. If $G$ is nilpotent and $\Phi(H)=1$ then $H$ is a direct product of groups with pairwise coprime orders that are either $P$-groups or
elementary abelian groups.
\end{thm}


\vskip 0.4 true cm

\begin{center}{\textbf{Acknowledgments}}
\end{center}
The authors are partially supported by the ``National Group for Algebraic and Geometric Structures, and their Applications" (GNSAGA - INdAM).
\vskip 0.4 true cm

\bibliographystyle{amsplain}
\bibliography{booksVC}

\vspace{1cm}

{\footnotesize \pn{\bf Valentina Grazian}\; \\ {Dipartimento di Matematica e Applicazioni}, {Universit\`a di Milano-Bicocca,} {Milano, Italy}\\
{\tt Email: valentina.grazian@unimib.it}\\

{\footnotesize \pn{\bf Andrea Lucchini}\; \\ {Dipartimento di Matematica}, {Universit\`a di Padova,} {Padova, Italy}\\
{\tt Email: lucchini@math.unipd.it}\\

{\footnotesize \pn{\bf Carmine Monetta}\; \\ {Dipartimento di Matematica}, {Universit\`a di Salerno,} {Salerno, Italy}\\
{\tt Email: cmonetta@unisa.it}\\

\end{document}